\documentclass[ejs,preprint]{imsart}

\RequirePackage[OT1]{fontenc}
\RequirePackage{amsmath,amssymb,amsfonts,amsthm,bbm}
\RequirePackage[numbers,sort&compress]{natbib}
\RequirePackage[colorlinks,citecolor=blue,urlcolor=blue]{hyperref}
\usepackage{color,graphicx}
\graphicspath{ {./figures/} }

\arxiv{2008.01130}

\startlocaldefs

\hypersetup{linktocpage}

\numberwithin{equation}{section}
\graphicspath{ {./figures/} }

\newtheorem{theorem}{Theorem}%[section]

\newtheorem*{definition*}{Definition}
\newtheorem{corollary}{Corollary}
\newtheorem{lemma}{Lemma}
\newtheorem{proposition}{Proposition}

\newtheorem{remark}{Remark}

\newcommand{\PP}{\mathbb{P}}
\renewcommand{\P}{\mathbb{P}}
\newcommand{\RR}{\mathbb{R}}
\newcommand{\R}{\mathbb{R}}

\newcommand{\FF}{\mathbb{F}}
\newcommand{\F}{\mathbb{F}}
\newcommand{\G}{\mathcal{G}}
\newcommand{\X}{\mathcal{X}}
\newcommand{\W}{\mathbb{W}}
\newcommand{\A}{\mathcal{A}}

\renewcommand{\L}{\mathcal{L}}
\newcommand{\var}{\text{Var}}

\newcommand{\e}{\varepsilon}
\renewcommand{\d}{\delta}
\newcommand{\s}{\sigma}

\newcommand{\sumin}{\sum_{i=1}^n}
\newcommand{\Beta}{\mathop{\rm{Beta}}\nolimits}
\newcommand{\DP}{\mathop{\rm{DP}}\nolimits}
\def\generalprob#1{\ {\mathchoice{\raise-1.5pt\hbox{$\buildrel#1\over\ra$}}
{\raise-2pt\hbox{$\buildrel#1\over\ra$}}{}{}}\ }
\def\outeras{\generalprob{\small as*}}
\def\E{\mathord{\rm E}}
\def\Pr{\mathord{\rm P}}
\def\weak{\rightsquigarrow}
\def\ra{\rightarrow}
\mathchardef\given="626A

\endlocaldefs

\begin{document}

\begin{frontmatter}
\title{On the Bernstein-von Mises theorem for the Dirichlet process}
\runtitle{BvM for the Dirichlet process}

\begin{aug}
\author{\fnms{Kolyan} \snm{Ray}\thanksref{t2}\ead[label=e1]{kolyan.ray@imperial.ac.uk}}
\and
\author{\fnms{Aad} \snm{van der Vaart}\thanksref{t1,t2}\ead[label=e2]{avdvaart@math.leidenuniv.nl}}

\thankstext{t1}{The research leading to these result was (partly) financed by the NWO Spinoza prize awarded to A.W. van der Vaart by the Netherlands Organisation for Scientific Research (NWO).}
\thankstext{t2}{The research leading to these results has received funding from the European Research Council under ERC Grant Agreement 320637.}

\runauthor{K. Ray and A.W. van der Vaart}

\affiliation{King's College London and Universiteit Leiden}

\address{Department of Mathematics\\
Imperial College London\\
United Kingdom\\
\printead{e1}}
%\phantom{E-mail:\ }\printead*{e2}}

\address{Mathematical Institute\\
Leiden University\\ 
%P.O. Box 9512\\
%2300 RA Leiden\\
Netherlands\\
\printead{e2}}
%\printead{u1}}

\end{aug}

\begin{abstract}
We establish that Laplace transforms of the posterior Dirichlet process converge to those of the limiting Brownian bridge process in a neighbourhood about zero, uniformly over Glivenko-Cantelli function classes. For real-valued random variables and functions of bounded variation, we strengthen this result to hold for all real numbers. This last result is proved via an explicit strong approximation coupling inequality.
\end{abstract}



\end{frontmatter}

\section{Results}

Let $\PP_n=n^{-1}\sum_{i=1}^n \delta_{Z_i}$ be the empirical distribution of an i.i.d.\ sample $Z_1,\ldots,Z_n$ from a distribution
$P_0$ on some measurable space $(\X,\A)$, and given $Z_1,\ldots, Z_n$ let $P_n$ be 
a draw from the Dirichlet process with base measure $\nu+ n\PP_n$. Thus $\nu$ is a finite measure on
the sample space and $P_n\given Z_1,\dots,Z_n \sim \DP(\nu + n\PP_n)$ for all $n$, which is the
posterior distribution obtained when equipping the distribution of the observations
$Z_1,Z_2,\ldots, Z_n$ with a Dirichlet process prior with base measure $\nu$.
The case $\nu=0$ is allowed; the process $P_n$ is then known as the Bayesian bootstrap. For full definitions and properties, see the review in Chapter 4 of \cite{vandervaartbook2017}.

The Dirichlet process is the standard ``nonparametric prior" on the set of probability distributions on a (Polish) sample space and was first made popular in Bayesian nonparametrics by Ferguson \cite{Ferguson1974} and has subsequently been used in numerous statistical applications. The purpose of this note is to prove the following result concerning the Bernstein-von Mises theorem for the Dirichlet process posterior.

\begin{theorem}\label{thm}
Suppose $\G$ is a $P_0$-Glivenko-Cantelli class of measurable functions $g: \X\to\RR$ with
 measurable envelope function $G$, such that $\nu e^{tG}\le Ce^{ct^2}$ for every $t>0$ and some $c,C>0$, 
and $P_0G^{2+\d}<\infty$ for some $\d>0$.
%Set $\s_g^2=P_0(g-P_0g)^2$. 
Then there exists a neighbourhood of $0$, such that for every $t$ in the neighbourhood, 
\begin{align}
\label{EqMain}
\sup_{g\in \G}\left| \E\bigl[e^{t\sqrt{n}(P_ng-\PP_ng)} \given Z_1,\ldots,Z_n\bigr] - e^{t^2 P_0(g-P_0g)^2/2} \right|  \outeras 0.
\end{align}
\end{theorem}

Here we write $S_n\outeras0$ if there exist measurable random variables $\Delta_n$ with $|S_n|\le \Delta_n$
and $\Delta_n\ra0$ for $P_0^\infty$-almost every  sequence $Z_1,Z_2,\dots$. 
This is because the supremum in the theorem
is not necessarily measurable in $Z_1,Z_2,\ldots$, and, when it is not, convergence on a set of
non-measurability one may be less informative than it appears, as pointed out by a referee. 
(Cf.\ Chapter~1.9 in \cite{vandervaart1996} for further discussion.)

The function $t\mapsto e^{t^2\s^2/2}$ is the Laplace transform of the normal distribution with mean
0 and variance $\s^2$. The theorem thus says that the Laplace transform of the posterior
Dirichlet process centered at the empirical measure tends to the Laplace transform of a centered
normal distribution with variance $P_0(g-P_0g)^2$ in a neighbourhood of 0. This implies that the
posterior Dirichlet process tends in distribution to a normal distribution (see Section \ref{sec:laplace_conv}), which is a version of the Bernstein-von Mises theorem for
the Dirichlet process prior (a weak version, as the usual theorem gives the approximation in the
total variation distance; see Section~12.2 of \cite{vandervaartbook2017} for discussion). The
convergence of the Laplace transform is useful for handling for instance moments of the posterior
distribution.

The main contribution of the theorem is, however, to provide uniformity in a class
of functions $g$. This uniformity refers to the \emph{marginal} posterior distributions of the
process $\bigl(\sqrt{n}(P_ng-\PP_ng): g\in\G\bigr)$. The stronger sense of uniformity of
distributional convergence of this process as a random element in the set $\ell^\infty(\G)$ is known
to be true if $\G$ is a Donsker class, as shown in \cite{James2008} (also see \cite{lo1983,lo1986}). This
is a much stronger property than Glivenko-Cantelli as assumed here.

\begin{remark}
The assumption in Theorem~\ref{thm} that the class $\G$ possesses an envelope function
$G$ with sub-Gaussian tails under the base measure of the Dirichlet process prior $Q\sim \DP(\nu)$
ensures that the distribution of $QG$ possesses a finite moment generating function. If $G\in\G$, this 
is clearly necessary (see also Lemma~\ref{LemmaExponentialMomentsBeta}(iii) below).
It results from the fact that although vanishing in the limit, the prior
remains present in the posterior for every $n$. 
In the case of the Bayesian bootstrap, which formally corresponds to taking $\nu=0$,
this condition can be omitted.
\end{remark}

\begin{remark}
Theorem~\ref{thm} can be extended to the assertion \eqref{EqMain} for a sequence $\G_n$ of
classes of measurable functions. Inspection of the proofs below shows that if the total mass
of the base measures remains bounded, then it suffices that these classes satisfy
$$\sup_{g\in \G_n\cup \G_n^2}| \PP_ng-P_0g| \outeras 0,\qquad \sup_{g\in\G_n} P_0g^2=O(1),$$
and possess envelope functions $G_n$ such that 
$\max_{1\le i\le n} G_n(Z_i)= o(\sqrt n/\log n)$, almost surely, 
and $\E e^{tQG_n/\sqrt n}\le e^{ct}$ for   $Q\sim \DP(\nu)$ and every $0<t<\sqrt n$. 
The last condition  is implied by uniform sub-Gaussianity of the envelopes $G_n$,
but it is of course also satisfied if $\|G_n\|_\infty\lesssim \sqrt n$ for every $n$. 
(If the total mass of the base measures increases to infinity, then this last condition
must be replaced by a condition that ensures the assumption of Lemma~\ref{LemmaExponentialMomentsBeta}(iii).)
For convergence in probability in 
\eqref{EqMain}, it suffices that these conditions hold in probability.
%, and the  condition on the maximum is implied by the condition on the envelope. 
If the classes $\G_n$ are separable, then uniformity over $\G_n^2$ is implied by uniformity over
$\G_n$, as shown by Lemma~11 of \cite{rayvdv2018}. 
\end{remark}

Major applications of studying posterior Laplace transforms of functionals as in \eqref{EqMain} include establishing semiparametric and nonparametric Bernstein-von Mises theorems \cite{castillo2014b,castillo2015,ray2017,rivoirard2012}, especially for inverse problems \cite{monard2017,nickl2017,nickl2019}, posterior contraction rates in the supremum norm \cite{castillo2014,castillo2020,nicklray2019} and convergence rates for Tikhonov-type penalised least squares estimators \cite{monard2017,nicklray2019}. Such proofs typically require uniformity over function classes as established in \eqref{EqMain} and use likelihood expansions based on local asymptotic normality (LAN) of the model. Because the Dirichlet process prior does not give probability one to a dominated set of measures, the resulting posterior distribution cannot be derived using Bayes’ formula; one cannot thus use the LAN approach of the aforementioned papers to prove \eqref{EqMain}. 

Our result is applicable when a Dirichlet process prior is assigned to some distributional component of the model, such as the covariate distribution in regression models with random design. For example, Theorem \ref{thm} has recently been applied to establish semiparametric Bernstein-von Mises results for estimating average treatment effects in causal inference problems \cite{rayszabo2019,rayvdv2018}. Indeed, results there suggest that for estimating functionals, using a Dirichlet process prior on the covariate distribution can yield better performance than other common priors choices, especially in high-dimensional covariate settings.

Bernstein-von Mises and Donsker-type results have also been obtained for generalizations of the Dirichlet process, such as P\'olya trees \cite{castillo2017} and the Pitman-Yor process \cite{James2008,franssen2021}, and similar style results to Theorem \ref{thm} can be expected to hold for such priors. While the main outline of our proof should be applicable, it relies on the explicit posterior representation \eqref{EqDP} of the Dirichlet process for several technical computations, which would have to be extended to the explicit posterior representations available for these other priors.

\subsection*{The case $\X = \R$}

The proof of Theorem \ref{thm} requires uniformly bounded exponential moments of the process $(\sqrt{n}(P_ng-\PP_ng):g\in \G)$, which only holds for small $|t|$ under the moment condition $P_0G^{2+\delta}<\infty$ of the theorem (see Lemmas \ref{LemmaExpMoments} and \ref{LemmaMoments}). When $\X = \R$, we can strengthen Theorem \ref{thm} to hold for all $t\in \R$ under significantly stronger conditions on $\G$.

We now assume $Z_1,Z_2,\dots$ are i.i.d. random variables taking values in $\X = \R$. Recall that the total variation of a function $f:\R \to \R$ on an interval $[a,b]$ is
$$V_a^b(f) = \sup_{\Pi \in {\cal P}_a^b} \sum_{i=1}^{n_{\Pi}} |f(x_i)-f(x_{i-1})|,$$
where ${\cal P}_a^b = \{ \Pi = (x_0,\dots,x_{n_P}): a=x_0\leq x_1 \leq \dots \leq x_{n_{\Pi}}, n_{\Pi} \in \mathbb{N}\}$ is the set of all partitions of $[a,b]$, and define $|f|_{BV} = \sup_{a,b} V_a^b(f)$. 
%(As shown in the proof of the following theorem, the function $f$ necessarily has limits at $\pm \infty$ and one can equivalently take 
%$|f|_{BV}$ the supremum of the variation of $f$ over all finite partitions of $[-\infty, \infty]$.)

\begin{proposition}\label{BV_prop}
Suppose $\G$ is a class of right-continuous functions $g: \R\to\RR$ such that $\sup_{g\in \G} |g|_{BV} <\infty$.
Then for every $t\in \R$, for $P_0^\infty$-almost every  sequence $Z_1,Z_2,\ldots$,
\begin{align*}
\sup_{g\in \G}\left| \E\bigl[e^{t\sqrt{n}(P_ng-\PP_ng)} \given Z_1,\ldots,Z_n\bigr] - e^{t^2 P_0(g-P_0g)^2/2} \right| \ra 0.
\end{align*}
\end{proposition}

Since bounded variation balls are universal Donsker classes, this is a significantly stronger requirement than $\G$ being $P_0$-Glivenko-Cantelli in Theorem \ref{thm}. We prove this result by exploiting a strong approximation, which establishes a rate of convergence for representations of these random variables defined on a common probability space and has various applications in probability and statistics, for instance studying distributional approximations of transformed random variables $\psi_n (\sqrt{n}(P_n-\P_n))$, where the functions $\psi_n$ depend on $n$. For an overview of the theory of strong approximations and a survey of their applications in probability and statistics, see Cs\"org\H{o} and R\'ev\'esz \cite{csorgo1981} and Cs\"org\H{o} and Hall \cite{csorgo1984}, respectively.

Let $F_n(t) = P_n1_{(-\infty,t]}$ and $\FF_n(t) = \PP_n 1_{(-\infty,t]}$ denote the distribution function of the posterior Dirichlet process draw $P_n$ and empirical distribution function, respectively. In a slight abuse of notation, we shall write $F\sim \DP(\nu)$ to mean $F=P1_{(-\infty,\cdot]}$ for $P\sim DP(\nu)$. We write $|\nu| = \nu(\R)$. Recall that a \textit{Brownian bridge} $\{B(s):s\in[0,1]\}$ is a mean-zero Gaussian process with covariance function $\E B(s_1)B(s_2) = s_1\wedge s_2-s_1s_2$. A \textit{Kiefer} process $\{K(s,t):s\in[0,1], t\geq 0\}$ is a two-parameter mean-zero Gaussian process with covariance function $\E K(s_1,t_1)K(s_2,t_2) = (s_1\wedge s_2-s_1s_2)(t_1\wedge t_2)$. For each $t>0$, $\{t^{-1/2}K(s,t):s\in[0,1]\}$ is a Brownian bridge, while $\{K(\cdot,n+1)-K(\cdot,n): n\geq1 \}$ is a sequence of independent Brownian bridges. %For further properties of $B$ and $K$, see \cite{csorgo1981}.

An almost sure strong approximation of the posterior Dirichlet process was established by Lo \cite{lo1987}. He showed that on a suitable probability space, there exist random elements $F$, $K$ and $Z_1,Z_2,\dots\sim^{iid}F_0$ such that $F| Z_1,\dots,Z_n \sim DP(\nu+n\P_n)$ for every $n$, $K$ is a Kiefer process independent of $Z_1,Z_2,\dots$ and such that
\begin{align}\label{eq:lo}
\sup_{z\in\R} \left| \sqrt{n}(F-\FF_n)(z) - \frac{K(F_0(z),n)}{n^{1/2}} \right| = O\Bigl(\frac{(\log n)^{1/2} (\log \log n)^{1/4}}{n^{1/4}}\Bigr) \quad\text{a.s.}
\end{align}
Applications of \eqref{eq:lo} include studying the large sample behaviour of the Bayesian bootstrap and smoothed Dirichlet process posterior \cite{lo1987}, as well as receiver operating characteristic (ROC) curves \cite{gu2008}. We revisit this result by establishing an explicit coupling inequality in order to make uniform the constants in \eqref{eq:lo}. This for instance allows control of exponential moments, which is needed to prove Proposition \ref{BV_prop}.

We henceforth assume that the underlying probability space is rich enough that all random variables and processes subsequently introduced may be defined on it. Since the posterior distribution is conditional on the observations $Z_1,\dots,Z_n$, it is natural for a Bayesian to index the Gaussian process in \eqref{eq:lo} by the empirical distribution function $\F_n$ to obtain a conditional Gaussian approximation. The following is the explicit coupling inequality analogue of Lemma 6.3 of \cite{lo1987}. 

\begin{theorem}\label{thm:str_approx_BB} On a suitable probability space, there exist i.i.d.\ random variables $Z_1,Z_2,\dots \sim^{iid} P_0$ and $F_n$, with $F_n\given Z_1,\dots,Z_n \sim \DP(\nu + n\PP_n)$, and a sequence of Brownian bridges $(B_n)$ independent of $Z_1,Z_2,\dots$, such that \begin{align*}
                                                                                                                                                                                                                                                                                                       &P \left( \left. \sup_{z\in \R} \left| \sqrt{n}(F_n-\FF_n)(z)  -  B_n(\FF_n(z)) \right| \geq \frac{C_1(\log n + |\nu|) + x}{\sqrt{n}}\right|Z_1,\dots,Z_n \right) \\
                                                                                                                                                                                                                                                                                                       &\hspace{10cm}\leq C_2e^{-C_3x}, \end{align*} for all $x>0$ and $n \geq 2$, where $C_1-C_3$ are universal constants.  \end{theorem}

This result says one can couple the Dirichlet process posteriors to a sequence of Brownian bridges independent of the underlying data. The theorem could also be rephrased with the random variables $Z_1,Z_2,\dots$ replaced by any real numbers $z_1,z_2,\dots$ to emphasize this independence.

For $x=x_n$ taken equal to a constant times $\log n$, the right side sums finite over $n$, 
and hence the complement of the events at $x_n$ are valid for every sufficiently large $n$, by the Borel-Cantelli lemma,
for almost every sequence $Z_1,Z_2,\ldots$.
Provided that $|\nu| = O(\log n)$, this yields that, for almost every sequence $Z_1,Z_2,\dots$,
$$\sup_{z\in \R} \left| \sqrt{n}(F-\FF_n)(z)  -  B_n(\FF_n(z)) \right|=O(n^{-1/2}\log n),$$
which improves on the rate $n^{-1/4}(\log n)^{1/2}(\log \log n)^{1/4}$ in Lemma~6.3 of \cite{lo1987}.
This is because we replace the KMT coupling used in \cite{lo1987}, which involves a Kiefer process, with a direct quantile coupling due to \cite{csorgo1978} involving dependent Brownian bridges. The following is the analogous result when the Brownian bridges are related amongst themselves by tying them to a Kiefer process $K(\cdot,n) = \sum_{i=1}^n B_i$ for $(B_i)$ independent Brownian bridges.

\begin{theorem}\label{thm:str_approx_KP}
On a suitable probability space, there exist i.i.d.\ random variables $Z_1,Z_2,\dots \sim^{iid} P_0$ and $F_n$, with 
$F_n\given Z_1,\dots,Z_n \sim \DP(\nu + n\PP_n)$, and a Kiefer process $K$ independent of $Z_1,Z_2,\dots$, such that
\begin{align*}
&P \Bigg( \sup_{z\in \R} \left| \sqrt{n}(F_n-\FF_n)(z)- n^{-1/2}K(\F_n(z),n) \right|  \\
&\qquad\qquad\qquad \geq  \frac{C_1|\nu| + x\log n}{n^{1/2}} + \frac{C_2\sqrt{\log n}\, x^{3/4}}{n^{1/4}}   \Bigg| Z_1,\dots,Z_n \Bigg) 
  \leq C_3e^{-C_4x},
\end{align*}
for all $x>0$ and $n \geq 2$, where $C_1-C_4$ are universal constants.
\end{theorem}

Theorem \ref{thm:str_approx_BB} does not say anything about the joint distribution in $n$ of the corresponding Brownian bridges and thus only ``in probability" or ``in distribution" limit results can be proved. On the other hand, despite the slower convergence rate, Theorem~\ref{thm:str_approx_KP} can be used to establish the almost sure limiting behaviour of statistics of interest based upon $\sqrt{n}(F_n-\FF_n)(z)$, for instance a law of the iterated logarithm.

If $|\nu| = O(n^{1/4}(\log n)^{5/4})$, the above yields $P(\cdot|Z_1,Z_2,\dots)-$almost sure order 
$n^{-1/4}(\log n)^{5/4}$, significantly slower 
than the rate in Theorem~\ref{thm:str_approx_BB}. In Theorem \ref{thm:str_approx_KP} we follow the approach of \cite{lo1987} of using the KMT coupling rather than a quantile coupling as in Theorem \ref{thm:str_approx_BB}. Indeed, up to logarithmic factors, a better rate is not obtainable for coupling a quantile process with a Kiefer process \cite{deheuvels1998}, as opposed to dependent Brownian bridges. We obtain a slightly slower almost sure rate than the $n^{-1/4}(\log n)^{1/2}(\log \log n)^{1/4}$ achieved in Lemma 6.3 of \cite{lo1987} due to technical arguments used to make the coupling non-asymptotic. 

We may also index the Brownian bridges by the true distribution function $F_0$ at the expense of a slower rate. The following is the coupling inequality analogue of Theorem~2.1 of \cite{lo1987}.

\begin{corollary}\label{cor:trueF0_BB}
On a suitable probability space, there exist random variables $Z_1,Z_2,\dots \sim^{iid} F_0$ and $F_n$, 
with $F_n\given Z_1,\dots,Z_n \sim \DP(\nu + n\PP_n)$, and a sequence of Brownian bridges $(B_n)$ independent of $Z_1,Z_2,\dots$, such that for any $y>0$, the event
$$A_{n,y}=\{ \sqrt{n} \|\F_n-F_0\|_\infty \leq y \}$$
satisfies $P(A_{n,y}) \geq 1 -2e^{-2y^2}$ and
\begin{align*}
&P \Bigg( \sup_{z\in \R} \left| \sqrt{n}(F_n-\F_n)(z) -  B_n(F_0(z)) \right| \geq \frac{C_1(\log n + |\nu|) + x}{n^{1/2}}\\
&\qquad\qquad\qquad\qquad + \frac{C_2 \sqrt{y}(\sqrt{\log n}+\sqrt{x})}{n^{1/4}}  \Bigg| Z_1,\dots,Z_n \Bigg)1_{A_{n,y}} 
 \leq C_3e^{-C_4x},
\end{align*}
for all $x>0$ and $n \geq 2$, where $C_1-C_4$ are universal constants.
\end{corollary}

The Bayesian interpretation is that there are events $(A_{n,y})$ of high $P_0^n$-probability depending only on the observations $Z_1,\dots,Z_n$ on which one can approximate the posterior Dirichlet process with a sequence of Brownian bridges independent of the underlying data. If $|\nu| = O(n^{-1/4}(\log n)^{3/4})$,
setting $y = \sqrt{\delta\log n}$ with $\delta>1/2$ gives that for $P_0^\infty$-almost every sequence $Z_1,Z_2,\dots,$ we have approximation rate $n^{-1/4}(\log n)^{3/4}$, $P(\cdot|Z_1,Z_2,\dots)-$almost surely. A similar, if more complicated, expression can be proved with the Brownian bridges $(B_n)$ replaced by the Kiefer process $K$, in particular yielding an almost sure rate $O(n^{-1/4}(\log n)^{5/4})$, for $P_0^\infty$-almost every sequence $Z_1,Z_2,\dots$.

\section{Proofs}

\subsection{Proof of Theorem~\ref{thm}}

For given $Z_1,Z_2,\ldots, $ the Dirichlet process posterior distribution can be represented in law as the
convex combination
\begin{align}\label{EqDP}
P_ng= V_n \, Qg + (1-V_n)  \frac{\sumin W_ig(Z_i)}{\sumin W_i},
\end{align}
where $V_n\sim \Beta(|\nu|,n)$, $Q\sim \DP(\nu)$, and $W_1,W_2,\ldots$ are i.i.d.\ exponential variables
with mean 1, all variables $V_n, Q,, W_1, W_2,\ldots$ independent. 
This follows for instance from Theorem~14.37 in \cite{vandervaartbook2017} (with $\s=0$),
or representation properties of the Dirichlet distribution as in Proposition~G.10 in the same reference.
For $\nu=0$, the variables $Q$ and
$V_n$ should be interpreted as 0, and the first term vanishes. 
With the notation $\bar{W}_n = n^{-1} \sum_{i=1}^n W_i$, some algebra gives 
\begin{align}
\sqrt n(P_ng-\PP_ng)&=\sqrt n V_n\Bigl(Q g-\frac{\sumin W_ig(Z_i)}{\sumin W_i}\Bigr)\nonumber\\
&\qquad\qquad\qquad+\frac1 {\bar W_n}\,\frac1{\sqrt n}\sumin (W_i-1)(g (Z_i)-\PP_ng).
\label{EqRepresentation}
\end{align}
We show that the first term is negligible and the second tends to a normal distribution.

The variable $V_n$ is of the order $1/n$ and the first term in brackets on the right side 
of \eqref{EqRepresentation} is bounded above by 
$Q |g|+\max_{1\le i\le n}|g(Z_i)|$.
Because $\E Q |g|=\nu |g|/|\nu|<\infty$, and  by Lemma~\ref{LemmaMaxima} (below),
since $P_0g^2<\infty$,
\begin{equation}
\label{EqMax}
\frac 1{\sqrt n}\max_{1\le i\le n} |g(Z_i)|\ra0,\qquad \text{a.s.},
\end{equation}
the first term on the right tends to zero, in distribution conditionally given almost every sequence
$Z_1,Z_2,\ldots$, as $n\ra\infty$. 

The leading factor $1/\bar W_n$ of the second term on the right of \eqref{EqRepresentation} 
tends to 1 almost surely, by the strong law of large numbers.
If $P_0g^2<\infty$, then $\PP_n g\ra P_0g$ and $\PP_n g^2\ra P_0g^2$, almost surely.
This may be used together with \eqref{EqMax} to show that, for every $\e>0$,
\begin{align*}
  \frac1n\sumin \E \Bigl[(W_i-1)^2 (g(Z_i)-\PP_ng)^2\given Z_1, Z_2,\ldots\Bigr]\ra P_0(g-P_0g)^2
=: \sigma_g^2,&\qquad\text{a.s.}\\
\frac1n\sumin \E \Bigl[(W_i-1)^2 (g(Z_i)-\PP_ng)^21_{|W_i-1||g(Z_i)-\PP_ng|>\e\sqrt n}\given Z_1, Z_2,\ldots\Bigr]
\ra 0&,\qquad\text{a.s.}
\end{align*}
The Lindeberg central limit theorem then gives that, given almost every sequence $Z_1,Z_2,\ldots$,
the sequence $n^{-1/2}\sumin (W_i-1)(g(Z_i)-\PP_ng)$ tends in distribution
to a $N(0,\s_g^2)$-distribution.

The preceding together with Slutsky's lemma give that
\begin{equation}
\label{EqWeakConvergenceSingleg}
\sqrt n(P_ng-\PP_ng)\given Z_1, Z_2,\ldots\weak N(0,\s_g^2),\qquad\text{a.s.}
\end{equation}
If we would know that the moment generating function of the variables on the left
were bounded, then this would imply convergence of exponential moments, and the
proposition would be proved for $\G=\{g\}$.
The approach to proving the proposition will be to strengthen first the preceding display to uniformity in $g$, and next
show that exponential moments of the variables on the left are suitably bounded.

For the uniformity, we use the assumption that $\G$ is Glivenko-Cantelli. This is not overly strong, 
and it may be not far off from necessary. Indeed, if the variables $W_i-1$ were
standard normal instead of exponential, and the leading factor
$1/\bar W_n$ were not present and $\nu=0$, then the conditional distribution of the left side of the
preceding display would be $N\bigl(0,\PP_n(g-\PP_ng)^2\bigr)$, and convergence
of these normal distributions to $N(0,\s_g^2)$ would imply the convergence
$\PP_n(g-\PP_ng)^2\ra \s_g^2$, uniformly in $g$ if the convergence in distribution were
uniform. This is close to the Glivenko-Cantelli property.

To take account of possible non-measurability of the supremum in Theorem~\ref{thm}, we use
an explicit bound on the distance between the distributions in 
\eqref{EqWeakConvergenceSingleg}. 

\begin{lemma}
For independent mean-zero random variables
$X_1,\ldots, X_n$,  the bounded Lipschitz distance between the law of $n^{-1/2}\sumin X_i$ and the mean-zero normal
distribution with the same variance satisfies, for any $0<\e<1$, and 
$H(u)=u+u^{1/3}$ and $H_0(u)=u^{1/4}\bigl(1+|\log u|^{1/2}\bigr)$,
\begin{align}
&d\Bigl( \L\Bigl(\frac 1{\sqrt n}\sumin X_i\Bigr), N\Bigl(0, \frac 1n\sumin \E X_i^2\Bigr)\Bigr)\nonumber\\
&\qquad\qquad\lesssim \e+ \frac 1{\e^2} H\Bigl(\frac 1n\sumin \E X_i^21_{|X_i|>\e\sqrt n}\Bigr)
+ H_0\Bigl(\frac \e n\sumin \E X_i^2\bigr)
\label{EqCLT}
\end{align}
\end{lemma}

\begin{proof}
We combine results by  \cite{Yurinskii78} and \cite{Dehling83}, as outlined in the proof of Proposition~A.5.2 in
\cite{vandervaart1996}, which gives the lemma for i.i.d.\ variables. The first step is 
to note that, for $Z_i=X_i1_{|X_i|\le \e\sqrt n}-\E X_i1_{|X_i|\le \e\sqrt n}$, by Chebychev's inequality,
%and the fact that $\e\sqrt n\E|X_i|1_{|X_i|\le \e\sqrt n}\le \E X_i^21_{|X_i|\le \e\sqrt n}$,
$$P\Bigl(\Bigl|\frac 1{\sqrt n}\sumin (X_i-Z_i)\Bigr|>\e\Bigr)
\le \frac1{\e^2 n} \sumin \E X_i^21_{|X_i|> \e\sqrt n}.$$
We can then use Strassen's theorem (see e.g.\ \cite{Dudley2002}, Theorem~11.6.2) to see that
the Prohorov distance between the laws of $n^{-1/2}\sumin X_i$ and $n^{-1/2}\sumin Z_i$ is bounded
above by the maximum of $\e$ and the right side of the last display. Next, we apply Theorem~B in 
\cite{Dehling83} to bound the Prohorov distance between the law of $n^{-1/2}\sumin Z_i$ and
the mean-zero normal distribution with the same variance by a multiple of $H_0 (\e n^{-1} \sum_{i=1}^n \E X_i^2 1_{X_i \leq \e\sqrt{n}})$, where we have used $\E |Z_i|^3 \leq 2\e \sqrt{n}\E X_i^2$. Since $\sup_{0<v\leq u} H_0(v) \leq CH_0(u)$ for an absolute constant $C>0$, this is bounded by a multiple of the third term of \eqref{EqCLT}.
Finally, the distance between this normal distribution and the normal distribution as in the lemma is bounded by $(n^{-1} \sumin \E X_i^21_{|X_i|>\e\sqrt n})^{1/3}$ using Lemma 2.1 of \cite{Dehling83}

These estimates use the Prohorov distance, but are also valid for the bounded Lipschitz
distance, which is bounded above by twice the
Prohorov distance (see \cite{Dudley2002}, Corollary~11.6.5). 
(Then we may also replace $u^{1/3}$ by $u^{1/2}$ in $H$.)
\end{proof}

The idea of the lemma is that the first and third
terms on the right can be made arbitrarily small by choice of $\e$ if $n^{-1}\sumin \E X_i^2$ remains bounded
(as $H_0(u)\ra 0$ as $u\downarrow 0$),
while for fixed $\e>0$ the middle term tends to zero as $n\ra\infty$ if the Lindeberg condition holds.

\begin{lemma}
\label{LemmaUniformWeakConvergence}
Suppose $\G$ is a $P_0$-Glivenko-Cantelli class of measurable functions $g: \X\to\RR$ 
with envelope function $G$ such that $\nu G<\infty$ and $P_0G^{2}<\infty$. 
Then the convergence in \eqref{EqWeakConvergenceSingleg} is uniform in $g\in\G$. More precisely,
for $d$ the bounded Lipschitz distance,
$$\sup_{g\in \G} d\Bigl(\L\bigl(\sqrt n(P_ng-\PP_ng)\given Z_1, Z_2,\ldots\bigr), N(0,\s_g^2)\Bigr)
\outeras 0.$$
\end{lemma}

\begin{proof}
%By Lemma \ref{LemmaMaxima} and the assumption that $P_0G^2 <\infty$, 
%$$\frac1{\sqrt n}\max_{1\le i\le n} G(Z_i)\outeras 0.$$
By the square-integrability of $G$ and 
conservation of the Glivenko-Cantelli property under continuous transformations 
(see \cite{vdVW00}), the set $\{g^2: g\in \G\}$ is also Glivenko-Cantelli. 
Thus, for $\s_{n,g}^2=\PP_n(g-\PP_ng)^2$,
\begin{align*}
%\sup_{g\in\G} |\PP_ng-P_0g|&\outeras 0,\\
%\sup_{g\in\G} |\PP_ng^2-P_0g^2|&\outeras 0,\\
\sup_{g\in G} |\s_{n,g}^2-\s_g^2|&\outeras 0.
\end{align*}
We now apply \eqref{EqCLT} to the variables 
$X_i=(W_i-1)\bigl(g(Z_i)-\PP_ng\bigr)$ conditionally given $Z_1, Z_2,\ldots$, and that $\sup_{0<v\leq u} H_0(v) \leq CH_0(u)$,
to see that
\begin{align*}
&d\Bigl(\L\Bigl(\frac 1{\sqrt n}\sumin (W_i-1)\bigl(g(Z_i)-\PP_ng\bigr)\given Z_1,Z_2,\ldots,\Bigr), N(0,\s_{n,g}^2)\Bigr)\\
&\qquad \lesssim\e+ \frac 1{\e^2} H\Bigl( \sup_g \s_{n,g}^2 \E (W_1-1)^21_{|W_1-1|2\max_{1\le i\le n}G(Z_i)>\e \sqrt n }\Bigr) \\
& \qquad \qquad
+H_0\Bigl(\e\sup_g\s_{n,g}^2\Bigr).
\end{align*}
Here, $\sup_g \s_{n,g}^2$ can be further bounded by $\PP_nG^2$, resulting in a uniform upper on the distance that is a measurable function of $Z_1,Z_2,\ldots$. The resulting measurable upper bound tends to zero almost surely, as 
$n\ra\infty$ followed by $\e\ra 0$, because $\max_{1\le i\le n} G(Z_i)=o(\sqrt n)$ almost surely by
Lemma~\ref{LemmaMaxima}  and $P_0 G^2<\infty$.
Because $d\bigl(N(0,\s^2),N(0,\tau^2)\bigr)\le \sqrt{|s^2-\tau^2|}$ for any $\s,\tau>0$,
the variances $\s_{n,g}^2$ in the left side can be replaced by their limits $\s_g^2$ by the preceding
display.

This shows  the convergence of the second term 
in \eqref{EqRepresentation}, without the denominator $\bar W_n$, to the claimed Gaussian limits.
Here $\bar W_n\ra 1$, almost surely, and the first term is bounded
 above by $\sqrt n V_n \bigl(Q G+ \max_{1\le i\le n}G(Z_i)\bigr)$, which 
tends to zero almost surely, by the argument given before. Since $d\bigl(\L(\mu+\s X),\L(X)\bigr)
\le |\mu|+|\s-1|\,\E |X|$, for any $\mu,\s$ and random variable $X$, the 
scaling by $\bar W_n$ and shifting by the first term does not change the Gaussian limits.
\end{proof}

\begin{lemma}\label{LemmaExpMoments}
Suppose that the conclusion of Lemma~\ref{LemmaUniformWeakConvergence} holds and for some
$T>0$, there exists a measurable random variable $U$ such that
\begin{equation}
\label{EqExponentialMoment}
\limsup_{n\ra\infty} \sup_{g\in\G} 
\E\bigl[e^{T\sqrt{n}(P_ng-\PP_ng)} \given Z_1,\dots,Z_n\bigr]\le U <\infty,\qquad\text{ a.s.}
\end{equation}
Then \eqref{EqMain} holds for $0\le t<T$. Furthermore, if \eqref{EqExponentialMoment} holds
for some $T<0$, then \eqref{EqMain} holds for $T< t\le 0$.
\end{lemma}

\begin{proof}
For a fixed $t>0$ and $M>0$, the function $h_M(x)= e^{tx}\wedge M$ is bounded and Lipschitz. 
Lemma \ref{LemmaUniformWeakConvergence} thus gives that, with $\E_Z$ denoting the
conditional expectation given $Z_1,Z_2,\ldots$, 
$$\sup_{g\in\G} \Bigl| \E_Z\bigl[e^{t\sqrt n(P_ng-\PP_ng)}\wedge M\bigr]
-\int h_M\, dN(0,\s_g^2)\Bigr|\outeras 0.$$
Since $\sup_{g\in\G}\s_g^2\le P_0G^2 <\infty$, we can choose $M$ such that 
$|\int h_M\, dN(0,\s_g^2) -e^{t^2\s_g^2/2}|$ is
arbitrary small, uniformly in $g\in\G$. Furthermore, 
\begin{align*}
\Big| \E_Z\bigl[e^{t\sqrt n(P_ng-\PP_ng)}\wedge M- e^{t\sqrt n(P_ng-\PP_ng)}\bigr] \Big|
&\leq \E_Z\bigl[e^{t\sqrt n(P_ng-\PP_ng)}1_{e^{t\sqrt n(P_ng-\PP_ng)}\ge M}\bigr]\\
&\le \frac{1}{M^{(T-t)/t}}\,\E_Z\bigl[  e^{T\sqrt n(P_ng-\PP_ng)}\bigr].
\end{align*}
For $t<T$ and sufficiently large $M$ and $n$, this is arbitrarily small, uniformly in $g\in\G$, by assumption
\eqref{EqExponentialMoment}.

The proof of the assertion with $T<0$ is similar (or replace $P_n-\PP_n$ by $\PP_n-P_n$ in the argument).
\end{proof}

\begin{lemma}\label{LemmaMoments}
If $\G$ has envelope function $G$ such that $\int e^{tG}\,d\nu\le Ce^{ct^2}$ for every $t>0$ and some $C,c>0$, and
$P_0G^{2+\d}<\infty$ for some $\d>0$, then \eqref{EqExponentialMoment} holds for every $T$
in a sufficiently small neighbourhood of $0$.
%\footnote{The moment condition should be improvable to something like $P_0\bigl[G\log (1+G)\bigr]^2<\infty$?.}
\end{lemma}

\begin{proof}
By the Cauchy-Schwarz inequality, $\E e^{T(Y_1+Y_2)}<\infty$ if $\E e^{2TY_i}<\infty$, for $i=1,2$.
Thus it suffices to prove that the two terms on the right side of \eqref{EqRepresentation}
both possess finite exponential moments that are bounded in $n$. 

Since $P_0G^2<\infty$, we have that $\e_n:=\max_{1\le i\le n} G(Z_i)/\sqrt n\ra 0$, 
almost surely by Lemma \ref{LemmaMaxima}. The absolute value of the first term of 
\eqref{EqRepresentation} is bounded above by $nV_n(n^{-1/2}Q G+\e_n)$, where
$\e_n$ tends to zero almost surely. Thus this term
has bounded exponential moments by Lemma~\ref{LemmaExponentialMomentsBeta}.

Next consider the second term on the right side of \eqref{EqRepresentation},
or equivalently, assume that $\nu=0$. The absolute value satisfies
$$\sqrt n|P_ng-\PP_ng|
\le \frac1 {\bar W_n}\frac1{\sqrt n}\sumin W_i|g(Z_i)-\PP_ng|
%\le \frac{n^{-1/2}\sumin W_i 2\max_{1\le i\le n} G(Z_i)}{\bar W_n}
\le 2\sqrt n\max_{1\le i\le n} G(Z_i)=2n\e_n.$$
Since $\E e^X = 1+\int_0^\infty P(X \geq x) e^x dx$, it follows that for $T>0$,
\begin{align*}
&\E_Z\bigl[e^{T\sqrt n(P_ng-\PP_ng)}\bigr]\\
&\quad=1+\int_0^{2T\e_n n}\Pr_Z\Bigl( \frac T{\sqrt n}\sumin (W_i-1)(g(Z_i)-\PP_ng)>x\bar W_n\Bigr)\,e^x\,dx\\
&\quad\le 1+ \int_0^\infty\Pr\bigl(\bar W_n<1-\sqrt{3x/n}\bigr)\,e^x\,dx\\
&\quad\qquad+\int_0^{2T\e_n n}\Pr_Z\Bigl(\frac T{\sqrt n}\sumin (W_i-1)(g(Z_i)-\PP_ng)>x(1-\sqrt{3x/n})\Bigr)
\,e^x\,dx.
\end{align*}
The probability in the first integral on the far right is bounded above by $e^{-3x/2}$, 
for every $x>0$, using \eqref{eq:Gamma_exp}. Thus the integral involving this term
is bounded above by $\int_0^\infty e^{-x/2}\,dx=2$. For $x$ in the integration range 
of the second integral on the far right,
the number $1-\sqrt{3x/n}$ is at least $1/2$ for large enough $n$. Hence the preceding display is
bounded above by 
\begin{align*}
&3+\int_0^\infty \Pr_Z\Bigl(\frac T{\sqrt n} \sumin (W_i-1)(g(Z_i)-\PP_ng)>\frac x2\Bigr)\,e^x\,dx\\
&\qquad\qquad\qquad\qquad=2+ \E_Z \bigl[e^{2T n^{-1/2}\sumin (W_i-1)(g(Z_i)-\PP_ng)}\bigr].
\end{align*}
It suffices to show that the last expectation is finite and bounded in $g\in\G$ for some $T>0$.
%\footnote{We do \emph{not} get all $T>0$ with this argument. If the result is true
%for all $T$, then the preceding argument misses out on this by separately bounding
%the contributions of the $E_i$ to numerator and denominator. Can we do better, perhaps
%by a representation of the weights true uniform spacings?}

Let $\psi_1(x)=e^x-1$, and let $\|\cdot\|_{\psi_1}$ be the corresponding Orlicz norm. Then,
by Proposition~A.1.6 and Lemma~2.2.2 in \cite{vandervaart1996}, with the norms interpreted conditionally
given $Z$,
\begin{align*}&\Bigl\|\frac 1{\sqrt n}\sumin (W_i-1)(g(Z_i)-\PP_ng)\Bigr\|_{\psi_1}\\
&\quad\lesssim \Bigl\|\frac 1{\sqrt n}\sumin (W_i-1)(g(Z_i)-\PP_ng) \Bigr\|_2
+\frac 1{\sqrt n}\Bigl\|\max_{1\le i\le n}|W_i-1| |g(Z_i)-\PP_ng|\Bigr\|_{\psi_1}\\
&\quad\lesssim \sqrt{\PP_n (g-\PP_n g)^2}
+\frac{\log n}{\sqrt n}\|W_1-1\|_{\psi_1}\max_{1\le i\le n} |g(Z_i)|\\
&\quad\lesssim \sqrt{\PP_nG^2}+\frac{\log n}{\sqrt n}\max_{1\le i\le n} G(Z_i).
\end{align*}
Under the condition $P_0G^{2+\d}<\infty$, the first term is bounded almost surely by the law of large numbers, while 
the second term tends to zero almost surely by Lemma \ref{LemmaMaxima}.

By the definition of the Orlicz norm, $\E e^{|Y|/C}\le 2$ for $C\ge \|Y\|_{\psi_1}$ and any random variable
$Y$. This concludes the proof for $T>0$. For $T<0$, we copy the preceding argument, but replace $W_i-1$
by $1-W_i$ and $T$ by $|T|$.
\end{proof}

\begin{lemma}
\label{LemmaMaxima}
If $Y_1,Y_2,\ldots$ are i.i.d.\ random variables with $\E |Y_i|^r<\infty$ for some $r>0$, then
$\max_{1\le i\le n}|Y_i|/n^{1/r}\ra 0$, almost surely.
\end{lemma}

\begin{proof}
For any $y>0$ we have 
$$\max_{1\le i\le n} \frac{|Y_i|^r}n\le \frac{y^r}n+\frac1n\sumin |Y_i|^r1_{|Y_i|>y}.$$
As $n\ra\infty$ the first term tends to zero, for fixed $y$, while the second
tends to $\E |Y_1|^r1_{|Y_1|>y}$ by the law of large numbers, and can be made
arbitrarily small by choice of $y$.
\end{proof}

\begin{lemma}
\label{LemmaExponentialMomentsBeta}
Suppose that $Q\sim \DP(\nu)$ and $V_n\sim \Beta(|\nu|,n)$ are independent, and
let $G$ be a nonnegative,  measurable function.
\begin{itemize}
\item[(i)] If $t_n\ra t\in [0,1)$, then $\E e^{nt_n V_n}\ra (1-t)^{-|\nu|}$.
In particular, $\E e^{nt_nV_n}\ra 1$ when $t_n\ra0$.
\item[(ii)] If $\int e^{tG}\,d\nu\le C e^{ct^2}$ for every $t>0$ and some $C, c>0$, then
$\E e^{t\sqrt n V_n QG}\ra1$, for every $t\in [0, 1/\sqrt c)$.
\item[(iii)] If $\E e^{t\sqrt n V_n QG}=O(1)$ for some $t>0$, then 
there exist $C,c>0$ such that $\int e^{tG}\,d\nu\le C e^{ct^2}$, for every $t>0$.
\end{itemize}
\end{lemma}

\begin{proof}
We have that 
$$\int_0^1 e^{ntv}v^{|\nu|-1}(1-v)^{n-1}\,dv=n^{-|\nu|}\int_0^n e^{tu}u^{|\nu|-1}(1-u/n)^{n-1}\,du.$$
The integrand is dominated by $e^{tu}u^{|\nu|-1}e^{-u(1-1/n)}$, which is uniformly integrable for sufficiently large $n$ and $t<1$.
Therefore, for fixed $t<1$, the integral times $n^{|\nu|}$ is asymptotic to
$\int_0^\infty u^{|\nu|-1}e^{-u(1-t)}\,du=(1-t)^{-|\nu|}\Gamma(|\nu|)$ by the dominated convergence
theorem. By the definition of the beta distribution, the expectation $\E e^{nt_n V_n}$ is the
quotient of two of the integrals as in the display, with $t=t_n$ and with $t=0$, respectively.
This concludes the proof of (i).

For (ii), we first note that $\E e^{tQ G}\le \E Q e^{tG}=\int e^{tG}\,d\nu/|\nu|$ by Jensen's inequality,
whence $\E e^{t QG}\lesssim e^{ct^2}$ by the assumption.
By the independence of $Q$ and $V_n$ and a coordinate substitution as under (i),
$$\E e^{t\sqrt n V_n QG}
%= \frac{ \int_0^1 \E^{t\sqrt n v QG}  v^{|\nu|-1}(1-v)^{n-1}\,dv}{\int_0^1 v^{|\nu|-1}(1-v)^{n-1}\,dv}
= \frac{ \int_0^n \E e^{(tu/\sqrt n)QG}  u^{|\nu|-1}(1-u/n)^{n-1}\,du}{\int_0^n u^{|\nu|-1}(1-u/n)^{n-1}\,du}.$$
The integrand in the numerator tends pointwise to $u^{|\nu|-1}e^{-u}$ and is dominated by a multiple of
$e^{ct^2u^2/n}u^{|\nu|-1}e^{-u(1-1/n)}\le u^{|\nu|-1}e^{-du}$ on $[0,n]$, for a constant
$d<1-ct^2$ and sufficiently large $n$. The denominator is as before. The integrals thus have the same limit.

For (iii), note that by the stick-breaking representation of the Dirichlet process, the
variable $QG$ is stochastically larger than $W G(\theta)$, 
for $W\sim \Beta(1,|\nu|)$ independent of $\theta\sim \nu/|\nu|$.
It follows that for any $t \geq 0$,
$$\E e^{t QG}\ge \E e^{tWG(\theta)} = \int_0^1\int e^{t wG}\,d\nu\, (1-w)^{|\nu|-1}\,dw\geq \psi(t/2) \frac{1}{|\nu|2^{|\nu|}}$$
for $\psi(t)=\int e^{tG}\,d\nu$. Then by the preceding calculations, for sufficiently large $n$,
\begin{align*}
\E e^{t\sqrt n V_n QG}&\geq \frac{1}{|\nu|2^{|\nu|+1}\Gamma(|\nu|)} \int_0^n\psi\bigl(tu/(2\sqrt n)\bigr)u^{|\nu|-1}(1-u/n)^{n-1}\,du\\
&\ge \frac{\psi(t\sqrt n/4)}{|\nu|2^{|\nu|+1}\Gamma(|\nu|)} \int_{n/2}^n u^{|\nu|-1}(1-u/n)^{n-1}\,du.
\end{align*}
The integral is bounded below by $n^{-1}2^{-n}$, whence $\psi(t\sqrt n/4)\leq C(\nu) e^{n}$ for 
large enough $n$,
if the left side of the display remains bounded in $n$. For sufficiently large $s$, there exists large enough $n$ such that 
$t\sqrt {n-1}/4<s\le t\sqrt n/4$, and $\psi(s)\le C(\nu) e^{n}\le C(\nu) e^{32s^2/t^2}$.
\end{proof}

\subsection{Proof of Proposition \ref{BV_prop}}

\begin{proof}[Proof of Proposition \ref{BV_prop}]
Because by assumption the variation of a function $g\in\G$ is bounded uniformly over all intervals $[a,b]$, the limits of $g(x)$ as
$x\rightarrow\pm\infty$ exist and are finite. (Indeed the values $|g(a)|$, for $a<0$, are bounded by $|g(0)|+V_a^0(g)\le |g(0)|+V$ and hence
every sequence $g(x_n)$ with $x_n\rightarrow-\infty$ has a converging subsequence. If there were two subsequences $x_n$ and $y_n$
with different limits, then these could without loss of generality be chosen alternating: $x_1\ge y_1\ge x_2\ge y_2\ge\cdots$
and the variation over the partitions containing $y_N,x_N, \ldots y_1,x_1$ would tend to infinity with $N$.) It can be seen that 
the variation of the extended function $g$ over $[-\infty,\infty]$ is the supremum of the variations over all intervals $[a,b]$, and hence is also finite.
In particular, the functions $g-g(-\infty)$ are uniformly bounded.  As shifting the functions by a constant does not change the claim of the proposition,
we can assume without loss of generality that $g(-\infty)=0$, and that the class $\G$ has a uniformly bounded envelope function $G$. 
We can then decompose $g$ as $g=g^+-g^-$, for right-continuous, nondecreasing functions $g^+,g^-: [-\infty,\infty]\to \RR$,
uniformly bounded by $2V$
(e.g.\  Section 6.3 of \cite{royden2010}). Let $dg=dg^+-dg^-$ be the corresponding signed (Stieltjes) measure, and
$|dg|=dg^++dg^-$ its total variation.

We work on the probability space from Theorem \ref{thm:str_approx_BB}. 
For $(B_n)$ the Brownian bridges in that theorem and $g\in \G$, set 
$$\W_n g=-\int B_n\circ \F_n\,dg.$$
It can be seen that given $Z_1,Z_2,\ldots$, the variable $\W_ng$ possesses a $N(0,\|g-\P_ng\|_{L^2(\P_n)}^2)$-distribution,
whence $\W_n$ is a $\P_n$-Brownian bridge process, indexed by $\G$.

The process $F=F_n$ in Theorem \ref{thm:str_approx_BB} is the distribution function of the posterior
Dirichlet process $P_n$. By partial integration, we have
$$(P_n-P)g=\int g\,d(F_n-\F_n)=-\int (F_n-\F_n)\,dg.$$
Writing $\Delta_n = \|\sqrt{n}(F_n-\F_n)-B_n\circ\F_n\|_\infty$, we thus find that, for every sequence $Z_1,Z_2,\dots$,
$$|\sqrt{n}(P_n - \P_n)g - \W_n(g)| =\Bigl| \int \bigl(\sqrt n(F_n-\F_n)-B_n\circ\F_n\bigr)\,dg\Bigr| \le 4V\Delta_n.$$
Since bounded variation balls are uniform Donsker classes, $\G$ is $P_0$-Glivenko-Cantelli. 
Thus to prove the proposition, by Lemmas \ref{LemmaUniformWeakConvergence} and \ref{LemmaExpMoments},
 we need show only that \eqref{EqExponentialMoment} holds for all $T\in \R$. Using the last display and Cauchy-Schwarz,
\begin{align*}
\sup_{g\in\G} \E_Z\bigl[e^{T\sqrt{n}(P_n-\PP_n)g} \bigr] & \leq \left( \E_Z\bigl[e^{8TV\Delta_n} \bigr] \right)^{1/2} 
\times \sup_{g\in \G}  \left( \E_Z\bigl[e^{2T\W_n(g)} \bigr] \right)^{1/2}.
\end{align*}
The first term converges to 1 as $n\to \infty$ for every sequence $Z_1,Z_2,\dots$ by Lemma \ref{lem:exp_sup} below. 
The second term equals $\sup_{g\in \G} e^{T^2 \P_n(g-\P_ng)^2} \le  e^{T^2 \P_nG^2}\to e^{T^2 P_0G^2} <\infty$, $P_0^\infty$-a.s.
This establishes \eqref{EqExponentialMoment} and completes the proof.
\end{proof}

%Fix $n\in \mathbb{N}$ and let $a=a_n=Z_{(1)}$ and $b=b_n=Z_{(n)}$ be the smallest and largest values in the sample $Z_1,\dots,Z_n$, respectively. Since $g\in \G$ is of bounded variation, it can be written on $[a,b]$ as $g(x) = g(a) + g^+(x) - g^-(x)$, where $g^+,g^-:[a,b]\to [0,\infty)$ are bounded non-negative, non-decreasing functions with $g^+(b)+g^-(b) = V_a^b(g)\leq |g|_{BV} \leq V$.

%For $m \in \mathbb{N}$ and $\ell = 0,\dots,\lfloor 2^m V \rfloor$, set $I_{\ell} = \{ x\in [a,b]: g^+(x) \in [\ell2^{-m},(\ell+1)2^{-m}) \}$. Note that since $0\leq g^+ \leq V$ is non-decreasing, we have the partition $[a,b] = \cup_{\ell=0}^{\lfloor 2^m V \rfloor} I_\ell$ and  $\sup(I_\ell) \leq \inf(I_{\ell+1})$ if these intervals are not empty. Let $s_0 = 0$ and recursively set $s_{\ell} = \inf \{ x \in I_{\ell}: x\geq s_{\ell-1} \}$. Define $g_m = \sum_{\ell=0}^{\lfloor 2^m V\rfloor} \ell 2^{-m} 1_{(s_\ell,s_{\ell+1}]}$ with $(s_\ell,s_{\ell+1}]$ taken to be the empty set if $s_\ell = s_{\ell+1}$ ($g_m$ is just the function $2^{-m} \lfloor 2^m g^+ \rfloor = \sum_{\ell=0}^{\lfloor 2^m V\rfloor} \ell 2^{-m} 1_{I_\ell}$

\subsection{Proofs of strong approximation results}

We recall some useful facts. For a centered Gaussian process $(G_t)_{t\in T}$ with countable index set $T$ satisfying $\sup_{t\in T}|G_t| < \infty$ (Borell's inequality - Theorem 7.1 of \cite{ledoux2001}):
\begin{align}\label{eq:Borell}
P\left(\sup_{t\in T} |G_t| \geq \E\sup_{t\in T}|G_t| + x\right) \leq e^{-\frac{x^2}{2\sigma^2}},
\end{align}
for every $x>0$, where $\sigma^2 = \sup_{t\in T} \E G_t^2<\infty$. Note that if $G$ has continuous sample paths and $T \subset \R$ is uncountable, \eqref{eq:Borell} still holds, since we may restrict the supremum to a countable skeleton of $T$.

For $X_\theta \sim \text{Gamma}(\theta,1)$, we have the pair of exponential inequalities 
\begin{align}\label{eq:Gamma_exp}
P(X_\theta > \theta + \sqrt{2\theta x} + x) \leq e^{-x}, \quad \quad P(X_\theta < \theta - \sqrt{2\theta x}) \leq e^{-x},
\end{align}
for every $x>0$, see p. 28-29 of \cite{boucheron2013}. We also denote by $P_{Z}$ the conditional probability given $Z_1,\dots,Z_n$.

\begin{proof}[Proof of Theorem \ref{thm:str_approx_BB}]
Recall that $F|Z_1,\dots,Z_n = P_n1_{(-\infty,\cdot]}$ and let $\bar{F}|Z_1,\dots,Z_n = \bar{P}_n1_{(-\infty,\cdot]}$ for $\bar{P}_n \sim \DP(n\P_n)$. Using the representation \eqref{EqDP}, conditionally on $Z_1,\dots,Z_n$,
$$\|\sqrt{n}(F-\F_n)-\sqrt{n}(\bar{F}-\F_n)\|_\infty = \sup_{t\in \R} \sqrt{n} \left| (P_n-\bar{P}_n)1_{(-\infty,t]} \right| \leq \sqrt{n}V_n,$$
where $V_n \sim \Beta(|\nu|,n)$ is independent of $\bar{F}$. The random variable $V_n$ is equal in distribution to $X/(X+Y_n)$, where $X \sim \text{Gamma}(|\nu|,1)$ and $Y_n \sim \text{Gamma}(n,1)$ are independent. Applying \eqref{eq:Gamma_exp} gives $P(Y_n < Cn) \leq e^{-x}$ for $C = 1-\sqrt{2/3} >1/6$ and any $0<x\leq n/3$, and then that $P(X/Y_n > 6n^{-1}(|\nu| + \sqrt{2|\nu|x} + x)) \leq 2e^{-x}$ for all $0<x\leq n/3$. For $x \geq n/3$, we have the trivial probability bound
$$P(X/(X+Y_n) \geq 6n^{-1}(|\nu| + \sqrt{2|\nu|x} + x)) \leq P( X/(X+Y_n) \geq 2) = 0.$$
Combining the above and using $2\sqrt{|\nu|x} \leq |\nu|+x$, 
\begin{align}\label{eq:DP to BB ineq}
P_{Z}(\|\sqrt{n}(F-\F_n) - \sqrt{n}(\bar{F}-\F_n) \|_\infty \geq 12n^{-1/2}(|\nu| + x)) \leq 2e^{-x},
\end{align}
for all $x > 0$. It therefore remains to show the desired exponential inequality with $\sqrt{n}(\bar{F}-\F_n)$ instead of $\sqrt{n}(F-\F_n)$.

Let $U_1,\dots,U_{n-1} \sim U(0,1)$ be i.i.d. and independent of $(Z_i)_{i\geq1}$ and denote the corresponding order statistics by $0=U_{(0)}<U_{(1)} < \dots < U_{(n-1)} < U_{(n)}=1$. For given $Z_1,\ldots,Z_n$, the Bayesian bootstrap posterior distribution can be represented in law as $\bar{P}_n = \sum_{i=1}^n (U_{(i)}-U_{(i-1)})\delta_{Z_{(i)}}$, where $Z_{(1)} \leq \dots \leq Z_{(n)}$ are the order statistics of the sample and we have used the exchangeability of $(U_{(i)}-U_{(i-1)}:1\leq i \leq n-1)$. This gives
\begin{align}\label{eq:BB_draw}
\bar{F}(z) = \sum_{i=1}^n (U_{(i)}-U_{(i-1)})1_{\{ Z_{(i)} \leq z\} }.
\end{align}
Define the empirical quantile function $Q_{n-1}(t)$ of the $U_i$'s by
\begin{align*}
Q_{n-1}(t) = U_{(i)} \quad \quad \text{if} \quad  \frac{i-1}{n-1} < t \leq \frac{i}{n-1},\quad i=1,2,\dots,n-1
\end{align*}
and set $q_{n-1}(t) = \sqrt{n-1}(Q_{n-1}(t)-t)$ to be the uniform quantile process. By Theorem 1 of Cs\"org\H{o} and R\'ev\'esz \cite{csorgo1978}, one can define for each $n$ a Brownian bridge $\{\tilde{B}_n(t):0\leq t \leq 1\}$ on the same probability space such that for all $x \geq 0$,
\begin{align}\label{eq:quantile_str_approx}
P \left( \sup_{0\leq t\leq 1} \left| q_n(t)-\tilde{B}_n(t) \right| \geq \frac{c_1 \log n + x}{\sqrt{n}}\right) \leq c_2e^{-c_3x},
\end{align}
where $c_1,c_2,c_3$ are universal constants. Since these Brownian bridges are constructed based on $(U_i)_{i\geq 1}$, which are independent of $(Z_i)_{i\geq1}$, they may also be taken to be independent of $(Z_i)_{i\geq1}$. Setting $B_n = \tilde{B}_{n-1}$ and following \cite{lo1987},
\begin{equation}\label{eq:Lo argument}
\begin{split}
&\|\sqrt{n}(\bar{F}-\F_n)-B_n(\F_n)\|_\infty = \max_{1 \leq i <n} \left| \sqrt{n} \left( U_{(i)} - \frac{i}{n} \right) - B_n(i/n) \right| \\
& \qquad\leq \max_{1\leq i <n} \left| \sqrt{n} \left( U_{(i)} - \frac{i}{n-1} \right) - B_n(i/n) \right| + n^{-1/2} \\
& \qquad\leq \sqrt{\frac{n}{n-1}} \max_{1\leq i <n} \left| \sqrt{n-1} \left( U_{(i)} - \frac{i}{n-1} \right) - B_n(i/(n-1)) \right| \\
& \qquad + \sqrt{\frac{n}{n-1}}\max_{1\leq i <n} \left| B_n(i/(n-1))  - B_n(i/n) \right| \\
& \qquad +  \left( \sqrt{\frac{n}{n-1}} - 1 \right) \max_{1\leq i <n} \left| B_n(i/n) \right| +  n^{-1/2} \\
& \qquad =: I_B + II_B + III_B + n^{-1/2}.
\end{split}
\end{equation}
We prove separate exponential inequalities for $I_B-III_B$. For $n\geq 2$,
\begin{align*}
I_B &\leq \sqrt{2} \max_{1 \leq i <n} \left| q_{n-1}(i/(n-1)) - \tilde{B}_{n-1}(i/(n-1)) \right| \\
&\leq \sqrt{2} \sup_{0\leq t\leq 1} \left| q_{n-1}(t)-\tilde{B}_{n-1}(t) \right|,
\end{align*}
and the required inequality follows from \eqref{eq:quantile_str_approx}.

Setting $V_i = V_i^{(n)} = B_n(i/(n-1)) - B_n(i/n)$, we have $II_B \leq \sqrt{2} \max_{1 \leq i <n} |V_i|$ for $n\geq 2$. Since $B_n$ is a Brownian bridge, $V_1,\dots,V_{n-1}$ are Gaussian random variables with $V_i \sim N(0, \tfrac{i}{n(n-1)}(1-\tfrac{i}{n(n-1)}))$. Thus $\var(V_i) \leq 1/n$ for all $i$, and so the standard Gaussian maximal inequality Lemma 2.3.4 of \cite{gine2016} yields $\E\max_{1\leq i <n} |V_i| \leq C\sqrt{\log n/n}$ for an absolute constant $C>0$. Applying Borell's inequality \eqref{eq:Borell}, for $x>0$,
\begin{align}\label{eq:II_B}
P \left( II_B \geq \frac{C\sqrt{\log n}+x}{\sqrt{n}} \right) \leq e^{-cx^2}.
\end{align}

For $III_B$, recall that for a Brownian bridge $B_n$, $P(\|B_n\|_\infty > x) \leq 2e^{-2x^2}$ for $x>0$ (Proposition 12.3.4 of \cite{Dudley2002}). Using the mean value theorem with $h(x) = (1-x)^{1/2}$, for some $\xi\in[0,1/n]$,
$$\sqrt{\frac{n}{n-1}}-1 =  \sqrt{\frac{n}{n-1}} \left( 1 - \sqrt{1-1/n}\right) = \sqrt{\frac{n}{n-1}} \frac{1}{2n\sqrt{1-\xi}} \leq \frac{1}{n}.$$
Therefore, $P(III_B \geq n^{-1}x)\leq 2e^{-2x^2}$ for $x>0$.

Combining the exponential inequalities for $I_B-III_B$ via a union bound and comparing the dominating terms,
\begin{align*}
P_{Z} \left( \|\sqrt{n}(\bar{F}-\F_n) - B_n(\F_n) \|_\infty \geq \frac{C_1\log n + x}{\sqrt{n}}  \right) \leq C_2 e^{-C_3x},
\end{align*}
for all $x>0$ and universal constants $C_1,C_2,C_3 >0$. Together with \eqref{eq:DP to BB ineq} this gives the result.
\end{proof}

\begin{proof}[Proof of Theorem \ref{thm:str_approx_KP}]
Using the exponential inequality \eqref{eq:DP to BB ineq}, we need show only the result with $\sqrt{n}(\bar{F}-\F_n)$ instead of $\sqrt{n}(F-\F_n)$, where $\bar{F}$ is defined in \eqref{eq:BB_draw}. Let $H_n(s) = \tfrac{1}{n}\sum_{i=1}^n 1_{\{U_i \leq s\}}$ be the empirical distribution function of the i.i.d. random variables $U_1,U_2,\dots \sim U(0,1)$ in \eqref{eq:BB_draw}. With $\alpha_n(s) = \sqrt{n}(H_n(s) -s)$, $s\in[0,1]$, the uniform empirical process, the KMT inequality (Theorem 4 of \cite{KMT}) implies that there exists a Kiefer process $\widetilde{K}(s,t)$ with
\begin{align}\label{eq:KMT}
P \left( \sup_{0\leq s \leq 1} \left| \alpha_n(s) - n^{-1/2} \widetilde{K}(s,n) \right| \geq \frac{C(\log n)^2 + x \log n}{\sqrt{n}} \right) \leq D e^{-cx}
\end{align}
for all $n\geq 1$, $x>0$ and universal constants $C,c,D>0$. We take as Kiefer process $K(s,t) =-\widetilde{K}(s,t)$.

Arguing as in \eqref{eq:Lo argument},
\begin{align*}
& \|\sqrt{n}(\bar{F}-\F_n)-n^{-1/2}K(\F_n,n)\|_\infty \\ %&= \max_{1 \leq i <n} \left| \sqrt{n} \left( U_{(i)} - \frac{i}{n} \right) - n^{-1/2} K(i/n,n) \right| \\
%& \leq \max_{1\leq i <n} \left| \sqrt{n} \left( U_{(i)} - \frac{i}{n-1} \right) - n^{-1/2} K(i/n,n) \right| + n^{-1/2} \\
& \leq \sqrt{\frac{n}{n-1}} \max_{1\leq i <n} \left| \sqrt{n-1} \left( U_{(i)} - \frac{i}{n-1} \right) - (n-1)^{-1/2} K\left(\frac{i}{n-1},n-1\right) \right| \\
& \qquad\qquad + \max_{1\leq i <n} \left| \frac{\sqrt{n}}{n-1} K\left(\frac{i}{n-1},n-1 \right)  - n^{-1/2}K(i/n,n) \right| + n^{-1/2} \\
& =: I_K + II_K + n^{-1/2}.
\end{align*}
We again establish separate exponential inequalities for $I_K$ and $II_K$. For $n \geq 2$,
\begin{align*}
I_K & \leq \sqrt{2} \max_{1\leq i <n} \left| -\alpha_{n-1}(U_{(i)}) - (n-1)^{-1/2} K(U_{(i)},n-1) \right| \\
& \qquad\qquad + \sqrt{\frac{2}{n-1}} \max_{1\leq i <n} \left|  K(U_{(i)},n-1)-K\left(\frac{i}{n-1},n-1 \right) \right| \\
& \leq \sqrt{2} \sup_{0\leq s \leq 1} |\alpha_{n-1}(s) - (n-1)^{-1/2}\widetilde{K}(s,n-1)| \\
& \qquad\qquad + \sqrt{2}(n-1)^{-1/2} \sup_{0\leq s \leq 1} |K(s,n-1)-K(H_{n-1}(s),n-1)|.
\end{align*}
For the first term, we use the KMT inequality \eqref{eq:KMT}. Since $\{(n-1)^{-1/2}K(s,n-1):s\in[0,1]\}$ is a Brownian bridge for each $n\geq 2$, we use the first inequality in Lemma \ref{lem:BBridge_sup} to deal with the second term. Together these yield
\begin{align*}
P_{Z} \left( I_K \geq \frac{C_1 (\log n)^2 + x\log n}{n^{1/2}} + \frac{C_2\sqrt{\log n} x^{3/4}}{n^{1/4}} \right) \leq C_3e^{-C_4x}
\end{align*}
for all $x>0$ and universal constants $C_1-C_4>0$.

For $n\geq 2$,
\begin{align*}
II_K & \leq  \frac{\sqrt{n}}{n-1} \max_{1\leq i <n} \left| K\left(\frac{i}{n-1},n-1 \right)  - K\left(\frac{i}{n},n-1\right) \right|  \\
& \qquad +  \frac{\sqrt{n}}{n-1} \max_{1\leq i <n} \left| K\left(i/n,n-1 \right)  - K(i/n,n) \right|  \\
&\qquad + \left(  \frac{\sqrt{n}}{n-1} - \frac{1}{\sqrt{n}}\right) \max_{1\leq i <n} \left| K(i/n,n) \right|\\
& = II_K^{(1)} + II_K^{(2)} + II_K^{(3)}.
\end{align*}
Using again that $\{(n-1)^{-1/2}K(s,n-1):s\in[0,1]\}$ is a Brownian bridge, $II_K^{(1)}$ is equal in distribution to $II_B$, for which we use the inequality \eqref{eq:II_B}. Similarly, $II_K^{(3)} = (n-1)^{-1} \max_{1\leq i <n} | n^{-1/2}K(i/n,n)|$, which is equal in distribution to $III_B$ up to a universal constant factor, and hence satisfies $P(II_K^{(3)} \geq Cn^{-1}x) \leq 2e^{-2x^2}$ for all $x>0$. For $n\geq 2$, we have $II_K^{(2)} \leq Cn^{-1/2}\max_{1\leq i < n}|X_i|$, where $X_i = K(i/n,n-1)-K(i/n,n)\sim N(0,\tfrac{i}{n}(1-\tfrac{i}{n}))$ satisfies $\var(X_i) \leq 1/4$. By Lemma 2.3.4 of \cite{gine2016}, $\E\max_{1\leq i < n}|X_i| \leq C \sqrt{\log n}$ and so by Borell's inequality \eqref{eq:Borell}, $P(II_K^{(2)} \geq n^{-1/2} (C\sqrt{\log n} +x) ) \leq e^{-2x^2}$ for a universal constant $C>0$ and all $x>0$. Together these give for all $x>0$,
\begin{align*}
P \left( II_K \geq \frac{C_1\sqrt{\log n} + \sqrt{x}}{n^{1/2}} \right) \leq C_2 e^{-C_3x}.
\end{align*}
Using the exponential inequalities for $I_K$ and $II_K$, a union bound and that $x^{1/2} \lesssim \log n + x$,
\begin{align*}
&P_{Z} \Bigl( \|\sqrt{n}(\bar{F}-\F_n) - n^{-1/2}K(\F_n,n)\|_\infty \geq \frac{C_1(\log n)^2}{n^{1/2}} \\
&\qquad\qquad\qquad\qquad\qquad+\frac{x\log n}{n^{1/2}} + \frac{C_2\sqrt{\log n}x^{3/4}}{n^{1/4}} \Bigr) \leq C_3e^{-C_4x},
\end{align*}
for all $x>0$ and universal constants $C_1-C_4>0$. The first term on the right-hand side dominates if and only if $x \leq D (\log n)^2/n^{1/3}$ for a universal constant $D>0$. For such $x$, the upper bound in the last display is bounded by $C_3 \exp(-C_4 (\log 2)^2/2^{1/3})$ for all $n \geq 2$, which can be made larger than 1 by taking $C_3$ universal and large enough. The last display is thus trivially satisfied for such $x$, which implies
\begin{align*}
P \left( \|\sqrt{n}(\bar{F}-\F_n) - n^{-1/2}K(\F_n,n)\|_\infty \geq \frac{x\log n}{n^{1/2}} + \frac{C_2\sqrt{\log n}x^{3/4}}{n^{1/4}} \right) \leq C_3e^{-C_4x},
\end{align*}
for all $x>0$ and (different) universal constants $C_2-C_4>0$. Together with \eqref{eq:DP to BB ineq} this yields the result.
\end{proof}

\begin{proof}[Proof of Corollary \ref{cor:trueF0_BB}]
That $P(A_{n,y}) \geq 1 - 2e^{-2y^2}$ follows from the Dvoretzky-Kiefer-Wolfowitz-Massart inequality. Let $\{B_n:n\geq 1\}$ be the Brownian bridges from Theorem \ref{thm:str_approx_BB}. By the triangle inequality,
\begin{align*}
\left|  \sqrt{n}(F-\F_n)(z)-  B_n(F_0(z)) \right| &\leq \left| \sqrt{n}(F-\F_n)(z) -  B_n(\F_n(z)) \right|\\
&\qquad\qquad +  \left| B_n(\F_n(z)) -B_n(F_0(z)) \right|,
\end{align*}
and the exponential inequality for the first term follows from Theorem \ref{thm:str_approx_BB}. Since $\{ B_n:n\geq 1\}$ are independent of $(Z_i)_{i \geq 1}$ by Theorem \ref{thm:str_approx_BB}, applying the second inequality in Lemma \ref{lem:BBridge_sup} gives
$$P \left( \left. \sup_{z\in\R} \left| B_{\F_n(z)} - B_{F_0(z)} \right| 1
\geq K \frac{\sqrt{y}\left(\sqrt{\log n}+\sqrt x\right)}{n^{1/4}} \right| Z_1,\dots,Z_n \right) 1_{A_{n,y}} \leq e^{-x},$$
for all $x>0$ and a universal constant $K>0$. The result follows by a union bound.
\end{proof}

\begin{lemma}\label{lem:BBridge_sup}
Let $B = \{B_t : t\in[0,1]\}$ be a Brownian bridge and $\F_n$ be the empirical distribution function of $Z_1,\dots,Z_n\sim^{iid} F_0$. 
Then there exists a universal constant $K>0$ such that, for $n\geq 2$ and every $x>0$,
\begin{align*}
P \left( \sup_{z\in\R} |B_{\F_n(z)} - B_{F_0(z)}| \geq K\frac{\sqrt{\log n}}{n^{1/4}}\, x^{3/4} \right) \leq 2e^{-x}.
\end{align*}
If $B$ is independent of $Z_1,\dots,Z_n$, then there also exists $K>0$ such that, for $n\geq 2$ and every $x>0$,
\begin{align*}
P \left( \left. \sup_{z\in\R} \left| B_{\F_n(z)} - B_{F_0(z)} \right| 
\geq K\|\F_n-F_0\|_\infty^{1/2} \left(\sqrt{\log n}+\sqrt x\right) \right| Z_1,\dots,Z_n \right) \leq e^{-x}.
\end{align*}
\end{lemma}
 
\begin{proof}
The intrinsic metric of the Brownian bridge is bounded above by the square root of the Euclidean distance, whence its
metric entropy integral is a multiple of $\delta\mapsto \delta \max\bigl(\sqrt{\log (1/\delta)},1\bigr)$. 
Therefore, by Dudley's theorem (see \cite{gine2016}, Theorem 2.3.8).
$\E \sup_{s,t}\bigl[|B_s-B_t|/J(|s-t|)\bigr]<\infty$, for 
$J(\delta)=\sqrt{\delta} \max\bigl(\sqrt{\log (1/\delta)},1\bigr)$. Because the process
$(s,t)\mapsto (B_s-B_t)/J(|s-t|)$ is centered Gaussian with uniformly bounded variance, 
we can apply Borell's inequality \eqref{eq:Borell} to see that there exist constants $D,E>0$ such that, for $y>0$,
$$\Pr\Bigl(\sup_{0<s,t<1}\frac{|B_s-B_t|}{J(|s-t|)}> E+y\Bigr)\le 2e^{-Dy^2}.$$
There exists a  constant $C>0$ such that $Cy^2\le D(y-E)^2$, for $y>2E$. Then, for $y>2E$,
$$\Pr\Bigl(\sup_{0<s,t<1}\frac{|B_s-B_t|}{J(|s-t|)}> y\Bigr)\le 2e^{-Cy^2}.$$
By making $C$ if necessary still smaller, we can ensure that the right side is bigger than 1 for $y\le 2E$,
and then the preceding inequality is valid for every $y>0$.

By the Dvoretsky-Kiefer-Wolfowitz-Massart inequality, we also have, for $y>0$,
$$\Pr\Bigl(\sup_{z\in\R}|\F_n(z)-F_0(z)|> y\Bigr)\le 2e^{-2ny^2}.$$
Combining these two inequalities, we see that, for every $y_1,y_2>0$,
$$\Pr \Bigl( \sup_{z\in\R} |B_{\F_n(z)} - B_{F_0(z)}| \geq y_1 J(y_2) \Bigr) \leq 2e^{-Cy_1^2}+2e^{-2ny_2^2}.$$
We  choose $y_1=\sqrt{2x/C}$ and $y_2=\sqrt{x/n}$ to reduce the right side to $4 e^{-2x}$,
and then have $y_1J(y_2)\ge K_1x^{3/4}\max(\sqrt{\log (n/x)},1)/n^{1/4}$, for some constant $K_1>0$. For $x<\log 2$,
we have that $2e^{-x}>1$ and hence the first inequality of the lemma is trivially satisfied. For $x\ge\log 2$,
we have $4 e^{-2x}\le 2e^{-x}$ and $\max(\sqrt{\log (n/x)},1)\ge K_2\sqrt{\log n}$, for some constant $K_2>0$ and $n\ge 2$.
The first inequality of the lemma follows.

For the second inequality of the lemma, note that 
$$\E_B[\sup_z|B_{\F_n(z)} - B_{F_0(z)}||Z_1,\dots,Z_n] \lesssim J\bigl(\|\F_n-F_0\|_\infty\bigr)=\sqrt{\|\F_n-F_0\|_\infty}\,\eta_n,$$ 
for $\eta_n^2= \max(\log (1/\|\F_n-F_0\|_\infty),1)$, by Dudley's bound, while 
$\sup_z\var_B( B_{\F_n(z)} - B_{F_0(z)}|Z_1,\dots,Z_n) \lesssim \|\F_n-F_0\|_\infty$. Therefore, by Borell's inequality \eqref{eq:Borell}, there exists $K>0$ such
that 
$$\Pr \Bigl( \Bigl. \sup_{z\in\R} \frac{| B_{\F_n(z)} - B_{F_0(z)} |}{\sqrt{\|\F_n-F_0\|_\infty}}
\ge K (\eta_n+y)\Bigr| Z_1,\ldots,Z_n \Bigr) \leq 2e^{-y^2}.$$
We conclude by noting that $\liminf_n \sqrt{2n \log \log n}\|\F_n - F_0\|_\infty = \pi/2>0$ a.s. by Mogulskii's law (p. 526 of \cite{ShorackWellner}), so that $\eta_n\lesssim \sqrt{\log n}$, a.s.
\end{proof}

\begin{lemma}\label{lem:exp_sup}
Consider the setting of Theorem \ref{thm:str_approx_BB} and let $\Delta_n = \|\sqrt{n}(F-\F_n)-B_n(\F_n)\|_\infty$. Then for any $t\in \R$ and every sequence $Z_1,Z_2,\dots$, as $n\to\infty$,
$$\E[e^{t\Delta_n}\given Z_1,\dots,Z_n] \to 1.$$
\end{lemma}

\begin{proof}
Suppose $t > 0$. For $\alpha_n =C_1(\log n+|\nu|)/\sqrt{n}$, with $C_1$ the universal constant from Theorem \ref{thm:str_approx_BB}, and using the change of variable $u=e^{t\alpha_n+tx/\sqrt{n}}$, 
\begin{align*}
&\E[e^{t\Delta_n}\given Z_1,\dots,Z_n]  \leq e^{t\alpha_n} + \int_{e^{t\alpha_n}}^\infty P(e^{t\Delta_n} \geq u | Z_1,\dots,Z_n )du\\
&\qquad = e^{t\alpha_n} +\frac{t e^{t\alpha_n}}{\sqrt{n}} \int_{0}^\infty P\left( \Delta_n \geq \frac{C_1(\log n+|\nu|) + x}{\sqrt{n}} \bigg| Z_1,\dots,Z_n \right) e^{tx/\sqrt{n}} dx.
\end{align*}
Using Theorem \ref{thm:str_approx_BB} and that $C_3 - t/\sqrt{n}>0$ for $n$ large enough,
\begin{align*}
\E[e^{t\Delta_n}|Z_1,\dots,Z_n] & \leq e^{t\alpha_n} +\frac{t e^{t\alpha_n}}{\sqrt{n}} \int_{0}^\infty C_2 e^{-(C_3-t/\sqrt{n})x}  dx\\
& = e^{t\alpha_n} \left( 1 + \frac{C_2 t}{C_3\sqrt{n} - t} \right) \to 1
\end{align*}
as $n\to \infty$. Since $\Delta_n \geq 0$ the lower bound $\E[e^{t\Delta_n}|Z_1,\dots,Z_n] \geq 1$ holds trivially, which completes the proof for $t>0$. The case $t<0$ follows by a similar argument.
\end{proof}

\subsection{Some weak convergence facts} \label{sec:laplace_conv}

For completeness, we include the proof that \eqref{EqMain} implies $\sqrt{n}(P_n - \P_n)g \weak N(0,P_0(g-P_0g)^2)$ for every $g\in \G$ and $P_0^\infty$-almost every sequence $Z_1,Z_2,\dots$.

\begin{lemma}
\label{LemmaLaplaceTransform}
If $Y_n$ are random variables with $\E e^{tY_n}\ra e^{t^2\s^2/2}$, for every $t$ in a subset of 
$\RR$ that contains both a strictly increasing sequence with limit 0 and a strictly decreasing sequence with limit 0,
then $Y_n\weak N(0,\s^2)$.%\footnote{It suffices to have at least 2 negative and 2 positive numbers and a sequence
%converging to 0.}
\end{lemma}

\begin{proof}
Let $\cal T$ be the set of points and let $a<0$ and $b>0$ be contained in $\cal T$.  Because $\E e^{tY_n}$
is bounded in $n$, for both $t=a$ and $t=b$, the sequence $Y_n$ is tight, by Markov's
inequality. For every $t\in \cal T$ strictly between $a$ and $b$, some power larger than 1 of the
variable $e^{tY_n}$ is bounded in $L_1$, and hence the sequence $e^{tY_n}$ is uniformly
integrable. Consequently, if $Y$ is a weak limit point of $Y_n$, then $\E e^{tY_n}$ tends to
$\E e^{tY}$ along the same subsequence for every $t\in (a,b)\cap \cal T$. In view of the assumption of
the lemma, it follows that $\E e^{tY}=e^{t^2\s^2/2}$. The set $t\in (a,b)\cap \cal T$ is infinite by
assumption. Finiteness of $\E e^{tY}$ on this set implies that the function $z\mapsto \E e^{z Y}$
is analytic in an open strip containing the real axis.  By analytic continuation it is equal to
$e^{z^2\s^2/2}$, whence $\E e^{is Y}=e^{-s^2\s^2/2}$, for every $s\in\RR$.
\end{proof}

\begin{corollary}
\label{CorLaplaceTransform}
If $(Y_n,Z_n)$ are random elements with $\E (e^{tY_n}\given Z_n)\ra e^{t^2\s^2/2}$, in
probability, for every $t$ in a set that contains both a strictly increasing sequence with limit 0 and a
strictly decreasing sequence with limit 0, then $Y_n\given Z_n\weak N(0,\s^2)$, in
probability.  If the convergence in the assumption is in the almost sure sense, then the
conclusion is also true in the almost sure sense.
\end{corollary}

\begin{proof}
  For the conclusion in probability it suffices to show that every subsequence of $\{n\}$ has a
  further subsequence with $d\bigl(\L(Y_n\given Z_n), N(0,\s^2)\bigr)\ra0$, almost surely, where $d$ is a metric defining weak convergence.  From the
  assumption we know that every subsequence has a further subsequence with
  $\E (e^{tY_n}\given Z_n)\ra e^{t^2\s^2/2}$, almost surely. For a countable set of $t$, we can
  construct a single subsequence with this property for every $t$, by a diagonalization scheme. The
  preceding lemma gives that $d\bigl(\L(Y_n\given Z_n), N(0,\s^2)\bigr)\ra0$, almost surely, along
  this subsequence.
\end{proof}

\noindent \textbf{Acknowledgements:} We would like to thank two referees for their helpful comments, in particular one referee for pointing out a missing step in the proof.

\bibliographystyle{acm}
\bibliography{BvM_DP_refs}{}

\end{document}